\documentclass[10pt,amssymb]{amsart}
\usepackage{graphicx}
\usepackage{amscd}
\usepackage{amsmath}
\usepackage{amssymb}
\usepackage{amsthm}
\usepackage{amsfonts}
\usepackage{fancyhdr}
\usepackage{mathrsfs}
\usepackage{amsxtra}
\usepackage{mathabx}
\newcommand{\sgn}{\operatorname{sgn}}

\newtheorem{theorem}{Theorem}
\newtheorem{lemma}{Lemma}

\newtheorem{proposition}{Proposition}

\newtheorem{observation}{Observation}
\def\beq{\begin{eqnarray*}}
\def\eeq{\end{eqnarray*}}
\def\bt{\begin{theorem}}
\def\et{\end{theorem}}
\def\bp{\begin{proposition}}
\def\ep{\end{proposition}}
\def\bl{\begin{lemma}}
\def\el{\end{lemma}}
\def\te{{\tilde{S}}}
\title{
 Uniform boundedness of the derivatives of meromorphic inner functions on the real line}
\author{Rishika Rupam}
\address{Texas A\&M University
\\ Department of Mathematics\\
College Station, TX 77843, USA}
\email{rishika@math.tamu.edu}
\begin{document}
\maketitle
\begin{abstract}
Inner functions are an important and popular object of study in the field of complex function theory. We look at meromorphic inner functions with a given spectrum and provide sufficient conditions for them to have uniformly bounded derivative on the real line. This question was first studied by Louis de Branges in 1968 and was later revived by Anton Baranov in 2011. 
\end{abstract}
\section{Introduction}

 An inner function on the upper half plane $\mathbb C_+$ is a bounded analytic function on $\mathbb C_+$ with unit modulus almost everywhere on the real line $\mathbb R$. A meromorphic inner function (MIF) on $\mathbb C_+$ is an inner function on $\mathbb C_+$ with a meromorphic continuation to $\mathbb C$. The spectrum of an MIF $\Theta$ is the level set $\{x \in \mathbb R: \Theta$ $(x)$ $= 1 \}$ and we denote it by $\sigma(\Theta)$. Inner functions arise often in the study of complex function theory. A rather well studied object is the Weyl-Titchmarsh inner function that frequently occurs in the study of the spectral theory of differential operators. We are interested in the following problem - Given a \textit{separated} sequence $\{a_n\}$ on $\mathbb R$, does there exist an MIF $\Theta$ with $\{a_n\}$ as spectrum, such that $|\Theta'|$ is uniformly bounded on $\mathbb R$? By a separated sequence $\{a_n\}$, we simply mean that there is a $\delta >0$ such that $|a_n -a_m| > \delta$, for all $n \neq m $ integers. 
 
 In his book 'Hilbert spaces of entire functions' \cite{DB}, Louis de Branges formulated a result (Lemma 16) that was equivalent to the the statement, 'Given any sequence of separated points $\{a_n\}$ on $\mathbb R$, there exists a meromorphic inner function, $\Theta$ such that $|\Theta'|$ is uniformly bounded on $\mathbb R$ and  $\sigma(\Theta)=\{a_n\}$.' In 2011, Anton Baranov discovered this statement to be false and demonstrated this in private communications with mathematicians working in this area \cite{BAR}. He noticed that any meromorphic inner function having the natural numbers $\mathbb N$  as spectrum must indeed have unbounded derivative on $\mathbb R$. In fact, he formulated a more general result which could be loosely stated as -- any MIF that has as spectrum- clusters followed by gaps must necessarily have unbounded derivative on $\mathbb R$. 
In this paper, we will characterize sequences for which there do exist corresponding MIFs with bounded derivatives, as well as describe the method used by Baranov to contruct counterexamples.

Before proceeding any further, we must clarify de Branges motivation for his result as well as its application. Lemma 16 that de Branges stated was used to show the existence of a non-zero measure $\mu$ that is supported on a sequence $\Lambda$, such that its Fourier transform $\hat\mu$ vanishes on an interval of positive measure. Readers may recognize this as Beurling's \textit{gap problem} for sequences, wherein he asks the question - under what conditions on the sequence $\Lambda$, does there exist a corresponding measure $\mu$ with $\hat\mu$ vanishing on an interval of positive measure?
In \cite{POLYA}, Mitkovski and Poltoratski provided a sufficient condition for Beurling's problem for separated sequences, in terms of the Beurling Malliavin density. We notice in hindsight that de Branges was specifically looking at sequences that were \textit{regular} and that these sequences satisfy the requirement as stated in \cite{POLYA}. Thus, despite the erroneous lemma, de Branges application of it still holds.  We describe this briefly in the applications below. We remark that even for such special sequences, however, there may not exist any corresponding MIF with a bounded derivative. This and other such counterexamples were constructed by Baranov and we decribe these in the last section. 

Apart from this, there have been demands for meromorphic inner functions with a certain spectrum and a bounded derivative in more general contexts. For instance, in the Beurling Malliavin theory for Toeplitz kernels, Makarov and Poltoratski require this to prove the most general form of the BM multiplier theorem \cite{MIF} (see application 2 below). In \cite{POLYA}, Mitkovski and Poltoratski have characterized P\'olya sequences and gap conditions using the existence of an MIF with bounded derivative. In his paper \cite{ABAR} on the stability of completeness of a system of exponentials under certain perturbations, Baranov requires the existence of such an MIF with spectrum as the perturbed sequence. He also states a sufficient condition (lemma 5.2) on sequences to possess this special MIF. We will exploit this as well as a new sufficient condition to describe sequences that are spectra for MIFs with bounded derivatives. We also prove a partial converse result. 
\subsection*{Acknowledgements} The author would like to thank Anton Baranov for going through this work in great detail and suggesting crucial corrections and improvements; to Mikhail Sodin and Mishko Mitkovski for extremely useful feedback; and finally her adviser Alexei Poltoratski for suggesting this problem and for innumerable discussions and good advice. 

\subsection{Preliminaries}A well known theorem by Riesz and V.I. Smirnov says that all meromorphic inner functions $\Theta$ have the form, \beq \Theta(z)= B_\Lambda (z) e^{iaz} \eeq
where $a\geq 0$ and $B_{\Lambda}$ is the Blaschke product of the zeros of the function given by $\Lambda = \{\lambda_n\}_n$, where $|\lambda_n| \rightarrow \infty$ and satisfy the convergence criterion,
\beq \sum_{\lambda_n \in \Lambda}  \frac{\Im \lambda_n}{1+ |\lambda_n|^2} < \infty. \eeq
These functions enjoy the antisymmetric relationship,
\begin{eqnarray} \label{antisymm}\Theta(z) = \frac{1}{\overline{\Theta({\overline{z}})}}, \end{eqnarray}
which can be obtained by Weierstrass's factorization theorem, as given in \cite{HAVIN}. Any meromorphic inner function $\Theta$ can be represented as $\Theta=e^{i\phi}$ on $\mathbb R$, where $\phi$ is an increasing real analytic function on $\mathbb R$. This is not hard to see this for $e^{iaz}$, $(a\geq 0)$ and for finite Blaschke products, which we can then generalize to infinite ones. 

 It is easy to construct a meromorphic inner function with a given spectrum. Although the procedure is standard in existing literature \cite{SAR}, we go over it again as the construction is crucial to our main proofs. 
\\ \indent
  Let $\displaystyle \{a_n\}_{-\infty}^{\infty}$ be a separated sequence on $\mathbb{R}$ (we proceed similarly for one-sided sequences also). Let $\mu$ be a Poisson finite, positive measure on $\mathbb {R}$ with point masses at the $a_n$, i.e., 
\begin{eqnarray} 
       \label{clark} \mu=\displaystyle \sum_{n=-\infty}^{\infty} w_n \delta_{a_n} 
\end{eqnarray}       
for some $w_n$ $>0$ such that $\displaystyle\sum_{n=-\infty}^{\infty} \frac {w_n}{1+ a_n^2} < \infty$.The Cauchy transform of a Poisson finite measure $\nu$ on $\mathbb R$ is given by
\beq K\nu (z) = \frac{1}{\pi i}\displaystyle\int_{\mathbb R} \bigg(\frac{1}{t-z} - \frac{t}{1+t^2}\bigg) d\nu(t). \eeq Applying the Cauchy transform to the measure $\mu$ just defined,  \begin{eqnarray}
     \nonumber K\mu(z) = \frac{1}{\pi i}\displaystyle \sum_{n=-\infty}^{\infty} \frac{w_n}{a_n-z}  - \frac{w_n a_n}{1+a_n^2},
    \end{eqnarray}
we have that $K\mu$ is an analytic funtion from the upper half plane $\mathbb{C_+}$ to the right half plane. \newline
We compose $K\mu$ with a fractional linear transformation that maps the right half plane into the unit disk to get $\Theta$ :$\mathbb C_+$ $\rightarrow$ $\mathbb D$ as follows,
\begin{eqnarray}
 \label{theta} \Theta(z) = \frac{K\mu(z)-1}{K\mu(z) +1}. 
\end{eqnarray}

Observe that $\Theta$ is a meromorphic inner function on $\mathbb{C_+}$, with spectrum the set, 
$\{a_n\}_{-\infty}^{\infty}$;  For 
$\mu$ is non negative, giving us $\Re K\mu(z) > 0$ on $\mathbb C_+$, with $K\mu(x) \in \mathbb R$ for all $x \in \mathbb R$, along with the fact that 
\begin{eqnarray}
 \nonumber w \rightarrow \frac{w-1}{w+1}
 \end{eqnarray}
 maps $\{\Re w>0 \}$ onto $\mathbb D$, taking $i\mathbb R$ onto the unit circle. Morever, we notice that $\Theta$ would take the value $1$ exactly at the singularities of $K\mu$, i.e. at the $a_n$s. \\ 
 \indent
 The measure $\mu$ is known as the Clark measure associated with the function $\Theta$. By a reversal of steps and using Herglotz's theorem one can construct a Clark measure given any meromorphic inner function on $\mathbb C_+$. In particular, Clark measures associated with inner functions are singular with respect to the Lebesgue measure, with the spectrum of the inner function as its support except, possibly, the point at infinity. A natural question to ask is if the inner function with spectrum $\{a_n\}$ is unique. A look at (\ref{clark}) assures us that that is quite not the case, for the $w_n$ are almost arbitrarily chosen. We can obtain restrictions on the $w_n$ by imposing additional conditions on the function. Here we ask for boundedness of the derivative on $\mathbb R$. 
 
 \subsection{Applications} Let us see how useful it is to have an MIF with a bounded derivative on $\mathbb R$ with two applications. 
   \begin{enumerate}
  \item  We describe the sufficient condition for the gap problem, as described in \cite{POLYA}- If $D_*({\Lambda})>0$, then there does exist a nonzero measure $\mu$, supported on $\Lambda$ such that $\hat{\mu}$ vanishes on an interval of positive length. Here $D_*(\Lambda)$ refers to the interior Beurling Malliavin density of $\Lambda$. There are several equivalent definitions of $D_*$. To understand the most relevant definitions here, let us clarify a few things. Given a separated sequence $\Lambda $ on $\mathbb R$, it's counting funtion $n_\Lambda$ is the step function that jumps by $1$ unit at each point in $\Lambda$ and is $0$ at $0$. 
    For $a>0$, a sequence $\Lambda$ is said to be \textit{a-regular} if \begin{equation*} \int_{\mathbb R} \frac{|n_{\Lambda}(x)-ax|}{1+x^2}dx < \infty.\end{equation*} 
        We now define the interior BM density as follows.
  \begin{equation*}D_*(\Lambda) = \sup\{a|\exists \hspace{5 pt}\mbox{an}\hspace{5 pt} a-regular \hspace{5 pt} \mbox{subsequence} \hspace{2 pt}\Lambda' \subset \Lambda  \}.\end{equation*} 
Using this definition, it is easy to see that \textit{a-regular} sequences have  interior density equal to $a>0$. Such sequences are also \textit{P\'{o}lya}, i.e., it has the following property- if there is an entire function $f$ of zero exponential type that is bounded on this sequence, then $f$ must be a constant. This was proved by de Branges in \cite{DB} and it's connection with the gap problem is described in \cite{POLYA}. Thus, for regular sequences, the gap condition holds.

The proof of the gap condition is much simpler in the case the there is an MIF with $\Lambda$ as spectrum and bounded derivative on $\mathbb R$. 

 Let us explain briefly de Branges motivation for having an MIF with bounded derivative. We refer the reader to  \cite{DB} for definitions and results. Consider an entire function $E$, satisfying the inequality
 \begin{equation*}|E(z)|>|E(\bar{z})|, \hspace{12 pt} \mbox{for } z \in \mathbb C_+.\end{equation*} Such functions are usually called de Branges functions. The de Branges space $B_E$ associated with $E$ is the space of entire functions $F$ such that $F(z)/E(z)$ and $F(z)/E^{\#}(z)$ are in $\mathcal H^2(\mathbb C_+)$, the Hardy space on $\mathbb C_+$. Here $E^{\#}(z):=\overline{E(\bar{z})}$. Every de Branges function $E$ gives rise to a meromorphic inner function $\Theta$ as follows.
 \begin{equation*} \Theta(z) := E^{\#}(z)/E(z). \end{equation*}
 Conversely, every MIF gives rise to a de Branges function (see, for instance, \cite{HAVIN}) . Thus, every MIF corresponds to a de Branges space of entire functions. Let $\Theta(t) = e^{i\phi(t)}$ on $\mathbb R$. The phase function of $E(z)$ is defined as $-\frac{1}{2} \phi(t)$. Now, we recall our discussion in the introduction, wherein we mentioned that de Branges required the gap condition to hold for a certain \textit{regular} sequence, i.e., he required the existence of a measure $\mu$ that was supported on this regular sequence such that it's Fourier transform $\hat\mu$ vanished on an interval of positive length. 
The crucial result that he used was the existence of a certain de Branges space of entire functions on $\mathbb C_+$ with some conditions on the mean type of the space. He was able to do this using the fact that $|\phi'|$ is uniformly bounded on $\mathbb R$. We refer the reader to 'Theorem 65' in \cite{DB} for details.

\item We will refer to \cite{MIF} for this application.
We use the standard notation $\mathcal{H}^p(=\mathcal{H}^p(\mathbb C_+))$ to denote Hardy spaces and $\mathcal{N^+}= \{G/H : G,H \in \mathcal{H}^{\infty}, H \hspace{0.05 in} \mbox{is outer}\}$ to denote the Smirnov-Nevanlinna class in $\mathbb C_+$. Given an inner function $\Theta$, we define model spaces in the Smirnov class and general Hardy spaces to be the spaces
\begin{eqnarray*}
K^+[\Theta]&=& \{F\in \mathcal N^+ \cap C^{\omega}(\mathbb R): \Theta\bar{F} \in \mathcal N^+ \},\\
K^p[\Theta]&=& K^+[\Theta] \cap L^p(\mathbb R).
\end{eqnarray*}
We recall that to every $U \in L^{\infty}(\mathbb R)$, there corresponds the Toeplitz operator $T_U: \mathcal H^2 \rightarrow \mathcal H^2$ defined as 
\begin{equation*} T_U(f) = P_+(Uf),\end{equation*}
where $P_+$ is the projection onto $\mathcal H^2$. The \textit{Toeplitz kernel} is defined as \begin{equation*} N[U]=\ker T_U.\end{equation*}
As with model spaces, we can also define the Toeplitz kernels in the Smirnov class and Hardy spaces as
\begin{eqnarray*}
N^+[U]&=& \{F \in \mathcal N^+ \cap L^1_{loc}(\mathbb R): \overline{UF} \in \mathcal N^+\}\\
N^p[U]&=&N^+[U]\cap L^p(\mathbb R), \hspace{0.2 in} (0<p\leq \infty).\end{eqnarray*}
As described in \cite{MIF}, these kernels play a crucial role in answering questions in the inverse spectral theory of differential operators and completeness problems of families of functions, among others. Let us consider one such problem.
Let $\Phi = e^{i\phi}$ be a meromorphic inner function and let $\Lambda \subset \mathbb R$. We say that $\Lambda$ is a \textit{defining set} for $\Phi$ if for any other meromorphic inner function $\tilde{\Phi}=e^{i\tilde{\phi}}$,
\begin{equation*} \tilde{\phi}=\phi \hspace{0.1 in} \mbox{on $\Lambda$} \hspace{0.1 in} \implies \Phi \equiv \tilde{\Phi}.\end{equation*}
It has been described in \cite{MIF} that a sufficient condition for $\Lambda$ to be a defining set for $\Phi$ is for $\Lambda$ to be a uniqueness set for $K^\infty[\Phi^2]$. This condition translates to one about Toeplitz kernels.
\begin{theorem}
$\Lambda$ is a uniqueness set for $K^{\infty}[\Theta]$ if and only if for every meromorphic inner function $J$ such that $\sigma(J)=\Lambda$, we have that $N^{\infty} [\overline{\Theta}J]={0}.$
\end{theorem}
Thus, the (non) triviality of Toeplitz kernels $N^{\infty}[\bar{\Theta}J]$ is useful in answering questions about defining sets. In particular, for MIFs that have a bounded derivative, these questions are easier to answer. Let us refer to the following two theorems,
\begin{theorem}
Suppose $\Theta$ is a tempered inner function. Then for any mermorphic inner function $J$ and any $p>0$,
\begin{equation*}
N^p[\bar\Theta J]\neq 0 \implies \exists n,\hspace{0.1 in} N^{\infty}[\bar{b}^n\bar{\Theta}J]\neq 0.
\end{equation*}
\end{theorem}

By $\Theta$ being \textit{tempered} one simply means that $\Theta'$ has at most polynomial growth at $\pm \infty$, i.e., $\exists N, \Theta'(x) = O(|x|^N)$, $x\rightarrow \infty$. Here, the $n$ we obtain in the theorem is the same as $N$, regarding the growth of $|\Theta'|$. Thus, in the case of bounded derivative $|\Theta'|$, we have that \begin{equation*}
N^p[\bar\Theta J]\neq 0 \implies  N^{\infty}[\bar{\Theta}J]\neq 0.
\end{equation*}
The next theorem comes under the heading of Beurling Malliavin multiplier theorem.
\begin{theorem}
Suppose $\Theta$ is a meromorphic inner function satisfying $|\Theta'| \leq const$. Then, for any meromorphic inner function $J$, we have 
\begin{equation*}N^+[\bar{\Theta}J]\neq 0 \implies \forall \epsilon, N^{\infty}[\bar{S}^{\epsilon}\bar{\Theta}J]\neq 0.\end{equation*}
\end{theorem}
The space $N^+$ contains $N^{\infty}$  and hence it is easier to construct functions in these spaces. Similarly, it is often easier to work with $N^2$, which is a subspace of the Hardy space $\mathcal H^2$. Using the above theorems, we can just restrict our attention to these larger spaces in case $\Theta$ has a bounded derivative. As mentioned in \cite{MIF}, in the case of a general bounded $\gamma$, where $\bar{\Theta}J=e^{i\gamma}$, we cannot multiply down to $H^{\infty}$, elements of $N[U](=N^2[U])$ even by using factors like $\bar{S}$. 

Thus, an MIF having a bounded derivative is an extremely useful object.
\end{enumerate}
 \section{Main results}
 We'll denote the gaps between the successive $a_n$s as, \begin{eqnarray}
 \label{delta} \Delta_n := \begin{cases}
                        &a_{n+1} -a_n \hspace{0.3 in}\forall n>0 \\
                        &a_n - a_{n-1} \hspace{0.3 in}\forall n\leq 0.
                        \end{cases}
 \end{eqnarray} 
 In their paper \cite{BM3}, Makarov and Poltoratski have proved the existence of the required inner function when the $\Delta_n$  are uniformly bounded.

 Our approach in this paper will be to consider sequences characterized by the growth of their gaps. We start with gaps that are increasing, but very slowly. Formally, the gaps obey the relation
 \beq \frac{\ln |a_n|}{\ln \ln \Delta_n} \lesssim \Delta_n \lesssim \ln |a_n|.\eeq
 Here, and throughout the paper, $f(n) \asymp g(n)$ will denote the existence of constants $c_1, c_2 >0$ such that $ c_1 f(n) \leq g(n) \leq c_2 f(n)$ for large enough $n$. And $f(n) \lesssim g(n)$ will mean $f(n) \leq c g(n)$ for some $c \geq 0$ and large enough $n$. 
 
 It turns out that this case is a generalization of the result proved in \cite{BM3}. 
     \begin{lemma}\label{small}
    If $\left\{a_n\right\}$ is a sequence in $\mathbb R$ and the $\Delta_n$, defined by (\ref{delta}) are such that 
    \begin{itemize}
     \item $\Delta_{n+1}$ $\asymp$ $\Delta_n$ \\
     \item $\frac{\ln |a_n|}{\ln \ln \Delta_n} \lesssim \Delta_n \lesssim \ln |a_n|$,
     \end{itemize} then there is a meromorphic inner function $\Theta$ on $\mathbb C_+$ such that  $\{a_n\}$ is the spectrum of $\Theta$ and $|\Theta'|$ is uniformly bounded.
    \end{lemma}

    Next, we consider sequences with slightly larger gaps. Baranov's counter example of the one sided sequence $\mathbb N$ leads us to ask the natural question- if we have $\mathbb N$ on one side, how sparse can the sequence be on the other side? Simple computations tell us that on the other side, the gaps may be at most geometrically increasing, i.e., $|\lambda_n| \lesssim e^{c|n|}$, for some $c>0$. To put it precisely, 
 \begin{observation}\label{obs}
 Let $\Theta$ be an MIF on $\mathbb C_+$ with uniformly bounded derivative on $\mathbb R$ and $\Lambda$ the spectrum of $\Theta$.
 If $\Lambda_{\pm}= \Lambda \cap \mathbb R_{\pm} $ and $\Lambda_{+}=\mathbb N$ ,  then $\exists$ a $c \geq 0$ such that 
 $|\lambda_n| \lesssim e^{c|n|}$ for $\lambda_n \in \Lambda_{-}$.
 \end{observation}
 \begin{proof}
 If $z_n = x_n + iy_n$ are the zeros of $\Theta=e^{i\theta}$, then  
 \begin{equation*} \theta'(x) = \sum_n \frac{y_n}{(x-x_n)^2 + y_n^2}.\end{equation*}
 We notice that these are sums of Poissons kernels, with the property that $\int_{\mathbb R} \frac{y_n}{(x-x_n)^2 + y_n^2} dx = \pi$. Let us restrict our attention to the zeroes in the upper right quadrant, i.e. $x_n>0$ and $y_n >0$. For any $t<0$ and fixed $x_n$, the integral $\int_{t}^{0} \frac{y_n}{(x-x_n)^2 + y_n^2} dx $ attains minimum at $y_n^2 = x_n(x_n-t)/t$ and increases for $y_n^2 \geq x_n(x_n-t)/t$. But we notice that $ \int_{t}^{0} \frac{y_n}{(x-x_n)^2 + y_n^2} dx \leq \int_{t}^{0} \theta'(x) dx = |\sigma(\Theta) \cap (t,0)|.$  So, when $|t|$ is large enough, then larger the $y_n$s,  the denser $\Lambda_{-}$. Thus, to explore the case when $\Lambda_-$ is as sparse as possible, we must assume that $y_n \leq K $ for all $n$, for some $K>0$.  On the other hand, the zeros of $\Theta$ must be bounded away from the real line in order for $|\Theta'|$ to be bounded. Thus, we can assume, without loss of generality, that $y_n \asymp 1$. Consider the entire function $E$ associated with $\Theta$, i.e. $\Theta(z) = E^{\#}(z)/ E(z) =  \overline{E(\overline{z})}/ E(z)$. These functions are called de Branges functions (associated with an MIF). We refer the reader to \cite{DB} and \cite{POLYA} for more on de Branges functions. We know that  $E - E^{\#}$ has zeroes on $\mathbb N$ (at least), and so must be of exponential type at least $\pi$. Thus, the exponential type of $E$ is at least $\pi$ and so the zeros of $E$ are of the form $z_n \asymp  n + iy_n$ where $n\in \mathbb N$. Then,
 \begin{equation*}
 |\sigma(\Theta) \cap (t,0)| \gtrsim \int_t^0 \sum_{\mathbb N} \frac{1}{(x-n)^2 + 1} dx \asymp \sum_{n \leq |t|} \frac{1}{n} \asymp \ln |t|.  
 \end{equation*}
 \end{proof}
 The extreme case in the above situation, i.e. when $|\lambda_n| \asymp e^{c|n|}$, has the property that the gaps are co-measurable, i.e., $\Delta_n \asymp \Delta_{n+1}$. This gives us motivation for our next result where we analyse a more general case of co-measurable gaps.   We recall that the choice of the weights $w_n$ determine the growth of the function. Let us choose $w_n \asymp \Delta_n$.

 \begin{lemma}\label{medium}
   If ${a_n}$ is a separated sequence on $\mathbb R$ and $\Delta_n$, defined as in (\ref{delta}) are such that 
   \begin{itemize}
    \item $\Delta_{n+1}$ $\asymp$ $\Delta_n$ and\\
    \item $\Delta_n \gtrsim (\ln|a_n|)^{2}$,      
    \end{itemize}
    then by choosing $w_n \asymp \Delta_n$, $\Theta$ defined as in (\ref{theta}) is a meromorphic inner function on $\mathbb{C_+}$ with spectrum $\{a_n\}_{-\infty}^{\infty}$ such that $|\Theta'|$ is uniformly bounded on $\mathbb R$.
    \end{lemma}
    
    Some examples of such sequences are $a_n=(\sgn n) |n|^k$, where $k >0$ and $a_n =(\sgn n) r^{|n|}$, where $r > 1$. 
    
    Next, we consider sequences that are sparse. By sparse, we mean sequences that are at least geometrically increasing with common ratio bigger than $1$, for instance $a_n= (\sgn n)e^{e^{|n|}}$. The following statement may seem technical, but all that is being said is: If we consider finite clusters of points such that consecutive clusters are sparse, then choosing the weights $w_n \asymp 1$, the corresponding inner function will have a bounded derivative on $\mathbb R$.

 \begin{lemma} \label{sparse}
       For each $n \in \mathbb N$, let us consider a finite sequence (cluster) of points $\{a_n^j\}_{1\leq j\leq m_n}$ ($m_n$ being uniformly bounded) that is defined by the property $\frac{a_n^{j+1}}{a_n^j} \rightarrow 1 $ as $n \rightarrow \infty$ and $1 \leq j \leq m_n$. Moreover, the gaps between consecutive clusters is large, in the sense that there is a $d>0$ such that  $\frac{a_{n+1}^j}{a_n^l} -1 >d>0$ for $n \geq 0$ and $\frac{a_{n-1}^j}{a_n^l} -1 >d>0$ for $n < 0$ . Then, there is a meromorphic inner function $\Theta$ such  $\sigma(\Theta)=\{a_n^j\}$ and $|\Theta'|$ is uniformly bounded on $\mathbb R$. 
   \end{lemma}
   To summarize, we have the  following
    \begin{theorem}
   Let $\left\{a_n\right\}$ be a separated sequence on $\mathbb R$ satisfying one of the conditions below,
   \begin{enumerate}
   \item $\Delta_{n+1}$ $\asymp$ $\Delta_n$ and $ \frac{\ln|a_n|}{\ln\ln|\Delta_n|} \lesssim \Delta_n \lesssim \ln|a_n|$  OR \\
   \item $\Delta_{n+1}$ $\asymp$ $\Delta_n$ and $\Delta_n \gtrsim (\ln |a_n|)^{2}$ OR \\
   \item there is a $d>0$ such that the sequence can be partitioned into clusters $\{a_n^j\}_n$, with number of points in each cluster being uniformly bounded, such that for any cluster, $\frac{a_n^j}{a_n^{j+1}} \rightarrow 1$ and between successive clusters :$\frac{a_{n+1}^j}{a_n^l}-1>d>0.$
   \end{enumerate}
   then there exists a meromorphic inner function with spectrum $\{a_n\}$ with uniformly bounded derivative on $\mathbb R$.
  \end{theorem}
  
The above results cover a wide range of sequences. What happens when the sequence falls in none of these categories? There are several ways ways in which this could happen and we have a partial converse, which is inspried by Baranov's counter example, as described in \cite{BAR}. Let us first state Baranov's result.
\begin{proposition}
Let $\{a_n\}$ be a separated sequence with the following property. Given any $N>0$, there is a cluster $\{a_n\}_{n=k}^{k+N}$ such that $a_{k+m} = a_k + m$ for $1\leq m \leq N$ and $a_{k+N+1}= a_{k+N}+N$. Then, any MIF with spectrum $\{a_n\}$ must have unbounded derivative on $\mathbb R$.
\end{proposition}
In essence, this sequence has arithmetic clusters followed by unbounded large gaps. We generalise this result as follows.
Let $\{a_n\}$ be a sequence of points on $\mathbb R$ and a $D>0$ be a constant such that given any $N>1$, there is a cluster of points $\{a_n\}_{n=1}^{N}$ such that $a_2 - a_1\geq N D$ and $a_{n+1} - a_n \leq D$ for $n\geq2$. Notice that this sequence has clusters whose size grows unboundedly (thus excluding case 3 from above) and $\Delta_1>N\Delta_2$ (thus the gaps are not co-measurable i.e., $\Delta_n \not\asymp \Delta_{n+1}$). Such sequences serve as counterexamples and we state the result below.

\begin{proposition}\label{counterexample} Suppose $\{s_n\}$ is a separated sequence on the real line and $D>0$ is a constant such that given any $N>0$, there is a subset $\{t_n\}_{n=1}^{N}$ such that $(t_1,t_N) \bigcap \{s_m\}= \{t_n\}_{n=1}^N$ for which $t_2-t_1 >N D$ and $t_{n+1}-t_n < D$ for all $2 \leq n\leq N-1$ and let $\Theta$ be an MIF with this spectrum $\{s_n\}$. Then given any $\delta>0$, there is a zero $z_n=x_n+iy_n$ of $\Theta$ such that $0<y_n < \delta$. Hence, $|\Theta'|$ is unbounded on $\mathbb R$. \end{proposition}
   
     The above result can be used to prove that even in the \textit{regular} case, we are not assured of an MIF with bounded derivative. We describe this in the last section.

     Another way to characterize non-sparsity would be to have growing clusters that are sparser than any arithmetic progression, with common ratio approaching $1$, intertwining with a subsequence having non-comeasurable gaps.

Some other cases that we still don't know about is when the gaps are co-measurable, $\Delta_n \asymp \Delta_{n+1}$ and
\begin{enumerate}
 \item The gaps are very small, $\Delta_n \lesssim \frac{\ln|a_n|}{\ln\ln \Delta_n}
 $. 
 \item The gaps are '\textit{in between}'
 i.e., $\ln|a_n|\lesssim\Delta_n\lesssim \ln^2 |a_n|$.
 \item The sequence has clusters and gaps, i.e. $\Delta_n \not\gtrsim (\ln|a_n|)^{2}$  and $\Delta_n \not\lesssim \ln|a_n|$. 
 \end{enumerate}

\section{Proofs and details}
As mentioned before, we closely follow the proof of the result in \cite{BM3} to give a proof of lemma \ref{small} 

\begin{proof}   
     We use Krein's shift formula to create a meromorphic inner function. 
     Define $b_n$:=$\frac{a_n+a_{n+1}}{2}$ and let $E:=\bigcup_n (a_n,b_n)$ to define the function 
     \begin{eqnarray}
      \label{kreins} \frac{1}{\pi i} \log \frac{\Theta +1}{\Theta -1} = Ku + ic, \mbox{   } u:=1_E -\frac{1}{2}, c\in \mathbb R
     \end{eqnarray}
     Let $\mu_1$ and $\mu_{-1}$ be the corresponding Aleksandrov-Clark's measures defined by the Herglotz representation 
     \begin{eqnarray}
      \nonumber \frac{1+\Theta}{1-\Theta} = K\mu_1 + const., \hspace{0.5 in} \frac{1-\Theta}{1+\Theta} = K\mu_{-1} +const.
      \end{eqnarray}
      The measures $\mu_1$, $\mu_{-1}$ have the following form:
        \begin{eqnarray}
         \nonumber \mu_1 = \displaystyle\sum_{n=-\infty}^{\infty} \alpha_n \delta_{a_n}, \hspace{0.5 in} \mu_{-1}= \displaystyle\sum_{n=-\infty}^{\infty} \beta_n \delta_{b_n},\\
        \end{eqnarray}
         for some positive numbers $\alpha_n$, $\beta_n$. We claim that 
         \begin{eqnarray}
          \label{al} \alpha_n \lesssim \Delta_n\ln \Delta_n, \hspace{0.5 in} \beta_n \lesssim \Delta_n\ln \Delta_n.
         \end{eqnarray}
     Since
     \begin{eqnarray}
      \nonumber |\Theta'| \asymp |1-\Theta|^2 |(K\mu_1)'|, \hspace{0.5 in} |\Theta'| \asymp |1+\Theta|^2 |(K\mu_{-1})'|,
     \end{eqnarray}
     we have 
     \begin{eqnarray}
      \nonumber \Theta'(x) \asymp \min \left\{ \sum \frac{\alpha_n}{(x-a_n)^2}, \sum \frac{\beta_n}{(x-b_n)^2} \right\}, 
     \end{eqnarray}
      It follows that if $x \in(a_m, a_{m+1})$, then by (\ref{al}),
       \begin{eqnarray}
        \nonumber |\Theta'(x)| \lesssim \int_{|t-x|\ge \Delta_m} \frac{\ln |t-x| dt}{(x-t)^2} \lesssim \frac{\ln \Delta_m}{\Delta_m} \lesssim 1. 
       \end{eqnarray}
        We will prove the estimate for $\alpha_n$s. The proof for $\beta_n$s is similar. 
        \begin{eqnarray}
         \nonumber \alpha_n = Res_{a_n}\bigg(\sum \frac{\alpha_n}{x-a_n}\bigg) &=& Res_{a_n}(K\mu_1) \\
         \nonumber &=& Res_{a_n} \bigg(\frac{1+\Theta}{1-\Theta}\bigg)\\
         \nonumber &=& const. Res_{a_n} e^{Ku},
         \end{eqnarray}
         where $u$ is as defined in (\ref{kreins}).
     
     \begin{eqnarray}
      \nonumber e^{Ku} &=& \exp\left\{ \int_{b_{n-1}}^{b_n} \frac{u(t) dt}{t-z} \right\}  \exp \left\{ \int_{\mathbb R \setminus(b_{n-1}, b_n)} \frac{u(t)dt}{t - z} \right\}\\
      \nonumber  &=& \exp\left\{ \int_{b_{n-1}}^{b_n} \frac{u(t) dt}{t-z} \right\} \exp \left\{ \int_{\mathbb R \setminus(b_{n-1}, b_n)} \frac{u(t)dt}{t - z} \right\}\\
       \nonumber &=& \frac{\sqrt{(b_n-z)(b_{n-1}-z)}}{a_n -z}  \exp \left\{ \int_{\mathbb R \setminus(b_{n-1}, b_n)} \frac{u(t)dt}{t - z} \right\}
      \end{eqnarray}
      Thus,
      %\nonumber \asymp \delta_n + \exp \left\{ \int_{\mathbb R \setminus(b_{n-1}, b_n)} \frac{u(t)dt}{t - a_n} \right\}
      \beq Res_{a_n} e^{Ku} \asymp \Delta_n \exp \left\{ \int_{\mathbb R \setminus(b_{n-1}, b_n)} \frac{u(t)dt}{t - a_n} \right\} \eeq
    Thus, it remains to estimate $\exp  \left\{ \int_{\mathbb R \setminus(b_{n-1}, b_n)} \frac{u(t)dt}{t - a_n} \right\} $. This is done as follows. 
    
    For $j > n$,
     \beq
    \displaystyle\int_{a_j}^{a_{j+1}} \frac{u(t)dt}{t-a_n} &=& \ln \frac{b_j-a_n}{a_j - a_n} -\ln\frac{a_{j+1}-a_n}{b_j - a_n} \\
    &=& \ln \bigg( 1+ \frac{\Delta_j}{a_j-a_n}\bigg) - \ln \bigg( 1+ \frac{\Delta_j}{b_j -a_n} \bigg) \\
    &=& \frac{\Delta_j}{a_j - a_n} - \frac{\Delta_j}{b_j - a_n} + O\bigg( \frac{\Delta_j^2}{(a_j - a_n)^2}\bigg) = O\bigg( \frac{\Delta_j^2}{(a_j - a_n)^2}\bigg) 
    \eeq
    Since we are on the positive real line,  we can take logs. 
    
    We have, 
    \beq \frac{\Delta_j^2}{(a_j - a_n)^2}\lesssim \displaystyle\int_{a_j}^{a_{j+1}} \frac{\ln t}{(t-a_n)^2}dt. \eeq
    Thus,
    \beq
    \displaystyle\sum_{j=n+1}^{\infty} \frac{\Delta_j^2}{(a_j - a_n)^2} &\leq& \displaystyle\int_{b_n}^{\infty} \frac{\ln t dt}{(t-a_n)^2} \\
    &\leq&  \displaystyle\int_{b_n}^{\infty} \frac{\ln (t-a_n) dt}{(t-a_n)^2} + \displaystyle\int_{b_n}^{\infty} \frac{\ln a_n dt}{(t-a_n)^2}\\
    &\lesssim& \frac{\ln \Delta_n}{\Delta_n} + \frac{\ln a_n}{\Delta_n}\\
    &\lesssim& \ln\ln|\Delta_n|,
    \eeq
    using integration by parts in the second step. 
    Thus,
    \beq \bigg|\displaystyle\int_{b_n}^{\infty} \frac{u(t) dt}{ t - a_n} \bigg| \lesssim \ln\ln \Delta_n.\eeq
Hence, we have obtained the estimate     
    \beq \alpha_n \lesssim \Delta_n e^{\ln\ln|\Delta_n|} = \Delta_n \ln \Delta_n.\eeq
   
            \end{proof} 
            
Let us now indulge in a simple observation that will aid us in proving as well as  understanding the proofs of lemmas \ref{medium} and \ref{sparse}. We recall the construction of the inner function as described in the first section. Using Cauchy's estimate, it is easy to see that if there is a strip around the real axis on which the function is uniformly bounded, then the derivative on the real line is also uniformly bounded. In other words, if there are constants  $c,m>0$ such that for $|\Im z| < c$, $|\Theta(z) | <m$, then 
 \beq |\Theta'(x)| \leq \frac{1}{2 \pi}\displaystyle\int_{|z-x| = c} \frac{|\Theta(z)|}{|z - x|^2} dz \leq \frac{m}{c}. \eeq 
 We recall (\ref{antisymm}) and the relationship of $\Theta$ with the Cauchy transform (\ref{theta}) to formulate a sufficient condition :
 \begin{observation}
 \label{obs2}
 If there exist constansts $c,m>0$ such that for $0< \Im z < c$, we have $|K\mu(z) -1| >m $, then the MIF $\Theta := (K\mu -1)/(K\mu +1)$ is such that $|\Theta'|$ is uniformly bounded on $\mathbb R$.
\end{observation}
Conversely, however, it is only required that there be a zero free strip for $\Theta$ about the real axis. In order to prove lemma \ref{medium}, we need the following result.
 \begin{lemma}\label{lem} 
 If $\Delta_{n+1}$ $\asymp$ $\Delta_n$ then choosing $w_n$ := $\Delta_n$ $\forall n$,  we have 
 \begin{eqnarray*} 
 \label{case1} \bigg|\displaystyle\sum_{n \neq k}  \bigg(\frac{w_n}{a_n-a_k}  - \frac{w_n a_n}{1+a_n^2}\bigg)\bigg|  \lesssim \ln |a_k|.
 \end{eqnarray*}
 \end{lemma}
 Let us see the effect of this result in the situation when the gaps in the sequence are at least logarithmically increasing.
 
    \begin{proof}(of lemma \ref{medium})
    Let $C$ be a constant such that $\Delta_n$ $\geq C \ln^2|a_n|$.
    Let us choose and fix a $\delta << C$.
    We will separate the real line into two disjoint sets:
\begin{enumerate}
\item $x \in (\frac{a_{k-1}+a_k}{2},\frac{a_k+a_{k+1}}{2} )$ and $|x-a_k|\geq \frac{\delta}{2}\frac{\Delta_k}{\ln|a_k|}$, \\
\item $x \in (\frac{a_{k-1}+a_k}{2},\frac{a_k+a_{k+1}}{2} )$ and $|x-a_k| < \frac{\delta}{2}\frac{\Delta_k}{\ln|a_k|}$.
\end{enumerate}
\textbf{Case 1}  
We take derivatives in (\ref{theta}) to obtain the estimate 
\beq |\Theta'(z)| \leq |1-\Theta|^2 \sum \frac{w_n}{|z-a_n|^2}. \eeq 
For any $x \in (\frac{a_{k-1}+a_k}{2},\frac{a_k+a_{k+1}}{2} )$ and $|x-a_k|\geq \frac{\delta}{2}\frac{\Delta_k}{\ln|a_k|}$, \begin{equation} \label{est1} |\Theta'(x)|\asymp \frac{w_k}{(x-a_k)^2} \leq  \frac{\Delta_k}{\delta^2\Delta_k^2/4(\ln|a_k|)^{2}} \leq \frac{4(\ln^2|a_k|)}{\delta^2\Delta_k} \lesssim 1. \end{equation}
\newline
\textbf{Case 2} We first notice that for 
$z \in D(b, r)$, where
\begin{equation*} b= \frac{1}{2}\bigg( \frac{a_k + a_{k+1}}{2}+\frac{a_{k-1} +a_k}{2}\bigg) \hspace{0.1 in}  \mbox{and}  \hspace{0.1 in}r= \bigg( \frac{ a_{k+1}-a_{k-1}}{4}\bigg),\end{equation*} we have that 
\begin{equation*}|K\mu(z) - K\mu(a_k)| \asymp 1.\end{equation*}
For, \begin{eqnarray*}
\bigg|\displaystyle\sum_{n\neq k} \bigg(\frac{w_n}{a_n -a_k} - \frac{w_n}{a_n -z}\bigg)\bigg| \leq \bigg|\displaystyle\sum_{n\neq k} 
\frac{w_n(a_k - z)}{(a_n -a_k)(a_n-z)}\bigg| &\lesssim& \sum_{n\neq k} \frac{\Delta_n \Delta_k}{|(a_n-a_k) (a_n-z)|} \\
&\asymp&\sum_{n \notin \{k-1,k,k+1\}}\frac{\Delta_n\Delta_k}{(a_n-a_k)^2} \\
&\asymp& \Delta_k \int_{\mathbb R \setminus (a_{k-1},a_{k+1})} \frac{dt}{(t-a_k)^2}\\
&\asymp& 1
\end{eqnarray*}

Let $x \in (\frac{a_{k-1}+a_k}{2},\frac{a_k+a_{k+1}}{2} )$ and $|x-a_k| \leq  \frac{\delta\Delta_k}{2\ln|a_k|}$, then for any $z \in D\bigg(x,\frac{\delta\Delta_k}{2(\ln|a_k|)}\bigg)$
\begin{eqnarray*}
|K\mu(z)| &\geq& \bigg|\frac{w_k}{a_k - z} - \frac{w_k a_k}{1+a_k^2} \bigg| -  \bigg|\displaystyle\sum_{i \neq k}  \frac{w_i}{a_i-z}  - \frac{w_i a_i}{1+a_i^2}\bigg| \\
&\geq& \bigg| \frac{\Delta_k}{\delta\Delta_k/2\ln |a_k|}  - \frac{\Delta_k}{a_k}\bigg| - C'\ln |a_k|\\ &\geq& \frac{\ln|a_k|}{2\delta} - C'\ln|a_k| + O(1).
\end{eqnarray*}
where $C'$ is such that $\bigg|\displaystyle\sum_{i \neq k}  \frac{w_i}{a_i-z}  - \frac{w_i a_i}{1+a_i^2}\bigg| \leq C' \ln|a_k|$, by lemma \ref{lem}. Thus, by choosing a sufficiently small $\delta$, we have that $K\mu$ is bounded away from $1$. We notice that $\delta$ is independent of $k$.
Thus $K\mu$ is large on disks centred at points \textit{close to} the $a_n$s. We recall obsertaion \ref{obs2} which stated that it is sufficient to have a strip above the real line on which $|K\mu|$ is bounded away from $1$. Here, we obtain a slightly weaker configuration - we have disks with centres at $a_k$ and radii $\frac{\delta\Delta_k}{2\ln|a_k|} \gtrsim 1$ such that at each point $z$ in the disk, $|K\mu|$ is bounded away from $1$. Thus, $|\Theta'(x)|$ is bounded for $x \in (\frac{a_{k-1}+a_k}{2},\frac{a_k+a_{k+1}}{2} )$ and $|x-a_k| \leq \frac{\Delta_k}{2\ln|a_k|}$.
Cases $1$ and $2$ together give us that $|\Theta'|$ is bounded on $\mathbb R$.

        \end{proof}

    Let's now prove Lemma (\ref{lem})

 \begin{proof}
  The proof is essentially computation of integrals. The underlying idea is that when $\Delta_{n+1}$ $\asymp$ $\Delta_n$ , the singular measure $\mu$, now with weight at $w_n$ equal to the gap $\Delta_n$ at $a_n$, behaves like the Lebesgue measure. Explicitly, we look at the following calculations. Let $n>k$, then 
  \begin{eqnarray}
   \nonumber \frac{w_n}{a_n - a_k} \lesssim \frac{\Delta_{n-1}}{a_n-a_k} \leq  \int_{a_{n-1}}^{a_n} \frac{dt}{t-a_k} \hspace{1 cm}\mbox{and} \hspace{1 cm} \frac{w_n a_n}{1+a_n^2} \geq \displaystyle\int_{a_{n-1}}^{a_n} \frac{t dt}{1+t^2}.
   \end{eqnarray}
  
 Thus,
 \beq
  \label{r1} 0\leq \bigg|\displaystyle \sum_{n=k+1}^{\infty} \bigg( \frac{w_n}{a_n-a_k} -\frac{w_n a_n}{1+a_n^2}\bigg) \bigg | \lesssim  \displaystyle\int_{a_k + \epsilon}^{\infty} \bigg(\frac{1}{t -a_k} - \frac{t}{1+t^2} \bigg) dt = \bigg(\ln|t-a_k| - 1/2\ln|1+t^2|\bigg)\bigg|_{a_k+\epsilon}^{\infty} \asymp \ln |a_k|,
  \eeq
  where $\epsilon$ is just some arbitrary positive number that is, say $> 1/2$.

 Identical calculations exist for the sum $\displaystyle\sum_{n < k } \bigg(\frac{w_n}{a_n-a_k}  - \frac{w_n a_n}{1+a_n^2}\bigg)$.
 
 \end{proof}

We now prove the following result leading to the proof of lemma \ref{sparse}. This lemma considers sparse singletons, which we will  generalize to sparse clusters.
 
 \begin{lemma} \label{singletons}
 Let $\left\{a_n\right\}$ be a sequence on $\mathbb R$ such that,
  $1-  \frac{a_{k}}{a_{k+1}}  $ $> d> 0$ $\forall k\geq 0$ and $ 1- \frac{a_k}{a_{k-1}}>d>0$ $\forall k <0$, where $d$ is independent of $k$. Then,
there is a meromorphic inner function on $\mathbb C_+$ with spectrum $\{a_n\}$ whose derivative is uniformly bounded in $\mathbb R$. 
   \end{lemma}

   \begin{proof}
   We will use lemma 5.2 in \cite{ABAR} to prove this result.
   First note that the ratio test for convergence of a series gives us that $\displaystyle \sum_{n=-\infty}^{\infty} \frac{1}{|a_n|} < \infty$. We also notice that for all $n \neq k$, 
   \begin{eqnarray}
    \label{e1} \bigg| \frac{a_k}{a_n} -1 \bigg| > \min \left\{d, \frac{d}{1-d}\right\} =: D .  
   \end{eqnarray}
   For, if $n >k$, then $1-\frac{a_k}{a_n} \geq 1-\frac{a_k}{a_{k+1}} > d$ and for $n<k$, $\frac{a_k}{a_n} -1 > \frac{d}{1-d}$. We notice that this also tells us that $\frac{a_k}{\Delta_k} <\frac{1}{D}$ for all $k$. Let us choose the weights $w_n = 1$. 
 
 We rearrange terms,
   \begin{eqnarray*}
    \nonumber \displaystyle\sum_{n \neq k} \bigg(\frac{w_n}{a_n -a_k} - \frac{a_n w_n}{1+a_n^2} \bigg)
     &=&\displaystyle\sum_{n \neq k} \frac{1 +  a_n^2 -  a_n^2 + a_n a_k}{(a_n-a_k)(1+a_n^2)}\\
    \nonumber &=&\displaystyle\sum_{n \neq k} \frac{1} {(a_n-a_k)(1+a_n^2)}+\displaystyle\sum_{n \neq k} \frac{ a_n a_k}{(a_k-a_n)(1+a_n^2)}\\
    \label{s1s2} &=&S_1 + S_2.
   \end{eqnarray*}
   Then,
   \begin{eqnarray*}
    \nonumber |S_1| \leq \displaystyle\sum_{n \neq k} \bigg|\frac{1} {(a_k-a_n)(1+a_n^2)}\bigg| \lesssim \displaystyle\sum_{n \neq k} \bigg|\frac{1}{a_n^2}\bigg| < \infty \\   \end{eqnarray*}
   and 
   \begin{eqnarray*}
   \nonumber |S_2| \leq \displaystyle\sum_{n \neq k} \bigg| \frac{a_na_k}{(a_k -a_n) (1+ a_n^2)} \bigg|  \leq \displaystyle\sum_{n \neq k} \left|\frac{a_k}{\Delta_k}\frac{a_n}{1+a_n^2}\right| \leq \frac{1}{D} \displaystyle\sum_{n \neq k} \bigg| \frac{1}{a_n} \bigg| < \infty. 
   \end{eqnarray*}
   Thus, we have that 
   \begin{equation*}
       \nonumber \sup_{n}\bigg|\displaystyle\sum_{n \neq k} \bigg(\frac{w_n}{a_n -a_k} - \frac{a_n w_n}{1+a_n^2}\bigg)\bigg| <\infty.
      \end{equation*}
  Hence, by lemma 5.2 in \cite{ABAR}, the corresponding MIF, defined by \ref{theta} has uniformly bounded derivative on $\mathbb R$.
  \end{proof}
      
      The hypothesis of the above lemma characterizes sequences which are sparse, i.e., at least geometrically increasing with common ratio strictly bigger than 1. Thus, gaps that grow rapidly (but are still finite) do indeed have the required inner function. Notice that we could make this result stronger by allowing sequences that, instead of singletons, have finite bunches that are sparsely distributed. For, each bunch would contribute a (uniformly) bounded weight to the existing sum. We prove lemma \ref{sparse}.
       
   \begin{proof}(of lemma \ref{sparse})
   Suppose we choose one point from each cluster and call  it $a_{m_0}^{j_0}$, then by the proof of the previous lemma,   
   \begin{equation} \nonumber \displaystyle\sum_{n \neq m_0} \bigg(\frac{w_n^j}{a_n^j -a_{m_0}^{j_0}} - \frac{a_n^j w_n^j}{1+(a_n^j)^2}\bigg) = \displaystyle\sum_{n \neq m_0} \bigg(\frac{1}{a_n^j -a_{m_0}^{j_0}} - \frac{a_n^j}{1+(a_n^j)^2}\bigg) < B,\end{equation} where $B$ is a bound, independent of $k$. 
   Consider a point $a_{n_0}^{j_0}$ and let the '*' in the sum denote summation over all points except $a_{m_0}^{j_0}$
    \begin{eqnarray*} \nonumber \bigg|\displaystyle\sum_* \frac{1}{a_n^j-a_{m_0}^{j_0}}  - \frac{a_n^j}{1+(a_n^{j})^2}\bigg| &=& 
      \bigg|\displaystyle\sum_{j\neq j_0} \frac{1}{a_{m_0}^j-a_{m_0}^{j_0} } - \frac{a_{m_0}^j}{1+(a_{m_0}^{j})^2} \bigg|+\bigg| \displaystyle\sum_{n \neq i_0, 1\leq j \leq n_m} \frac{1}{a_n^j-a_{m_0}^{j_0}}  - \frac{a_n^j}{1+(a_n^{j})^2}\bigg|\\ 
      &\leq& S + N\bigg|\displaystyle\sum_{n \neq k} \frac{1}{a_n -a_k} - \frac{a_n }{1+a_n^2}\bigg| \leq S+NB,
     \end{eqnarray*}    
    where $S$ and $B$ are constants. The maximum size of each cluster $N$ assures that $S$ is independent of $m_0$ and $n_0$ and we know from the previous lemma that $B$ is independent of $m_0$ and $n_0$.
            
         \end{proof}
          
We now proceed to the last part of our discussion. Before we begin our proof of proposition \ref{counterexample}, let us elucidate some notations. Let us enumerate the zeroes $z_n(=x_n + iy_n)$ of $\Theta$ and let  $\Theta(x)=e^{i\phi(x)}$ on $\mathbb R$. Let us pick and fix a large $N$  and let $\{t_i\}$ be a set of points on $\mathbb R$ as described in the statement of the lemma. Let $S$ be the box $(t_2,t_N) \times (0,\sqrt{ND})$ and $T$ the box $(t_1,t_2) \times (0,\sqrt{ND})$. Let $\tilde a$ be the mid point of the interval $(t_2,t_N)$ and let $\te$ be the box $(t_2,\tilde a) \times (0,\sqrt{ND})$. On the adjacent interval, let $\tilde c$ be the point in $(t_1,t_2)$ such that $\tilde{a}-t_2 = t_2 -\tilde{c}$. Since $z_n$ form the zeroes of the Blaschke product of $\Theta$, we can write
\begin{equation*}
\phi'(x) = \sum_n\frac{y_n}{(x-x_n)^2 + y_n^2}.
\end{equation*}
%\begin{figure}[ht]
%\begin{center}
%\includegraphics[height= 3in,width=4in]%{counter1.jpg}
%\end{center}
%\end{figure}
\begin{proof}
Suppose that the zeros are bounded away from the real line, i.e., there is a $\delta > 0$ such that $y_n \delta$ for all all the zeros $z_n = x_n + iy_n$ of $\Theta$. Without loss of generality, let $\delta = 1$. We have that for $t \in (\tilde c,t_2)$ and $s\in (t_2,\tilde{a})$, 
\begin{equation*} \label{ineq}
\sum_{z_n \notin S \cup T} \frac{y_n}{(s-x_n)^2+y_n^2} \leq \kappa \sum_{z_n \notin S \cup T}\frac{y_n}{(t-x_n)^2+y_n^2},
\end{equation*}
where $ \kappa >0$ is a constant independent of $N$.

Then, \begin{equation*}
\label{outside}\displaystyle\int_{t_2}^{\tilde a}\sum_{z_n \notin S \cup T} \frac{y_n}{(s-x_n)^2+y_n^2}dt \leq \kappa \displaystyle\int_{\tilde c}^{t_2}\sum_{z_n \notin S \cup T} \frac{y_n}{(t-x_n)^2+y_n^2}dt \leq \kappa \pi, \end{equation*}
and the zeros in the box $T=(t_1,t_2) \times (0, \sqrt{ND})$ induce the following inequality \begin{equation*} \label{t} \displaystyle\int_{t_2}^{\tilde a}\sum_{z_n \in  T} \frac{y_n}{(s-x_n)^2+y_n^2}ds \leq  \displaystyle\int_{\tilde c}^{t_2}\sum_{z_n \in  T} \frac{y_n}{(t-x_n)^2+y_n^2}dt \leq \pi . \end{equation*}
Let $Z$ be the number of zeros $\{z_n\}$ in the box $S$. Then,
\begin{equation*}
 \displaystyle\int_{t_2}^{\tilde{a}} \displaystyle\sum_{z_n \in S} \frac{y_n}{(s-x_n)^2 + y_n^2} \leq Z\pi. \end{equation*}
Then,
\begin{eqnarray*}
 \displaystyle\int_{t_2}^{\tilde{a}} \displaystyle\sum_{z_n \in S} \frac{y_n dt}{(s-x_n)^2 + y_n^2} &=& \displaystyle\int_{t_2}^{\tilde{a}}\bigg( \displaystyle\sum_{n \in \mathbb Z} -\sum_{ (S \cup T)^c}  - \displaystyle\sum_{T} \bigg) \frac{y_n}{(s-x_n)^2 + y_n^2}dt\\
 &>& N\pi  - \kappa \pi -\pi. \end{eqnarray*}
Thus,
\begin{eqnarray*}
\bigg(N - \kappa -1 \bigg)\pi \leq   Z\pi. 
\end{eqnarray*}
This gives us that 
\begin{equation*}
Z\geq N - \kappa -1.
\end{equation*}
Thus, for a large enough $N$, $Z \geq  N/2$. In a similar vein it can be proved that the box $\te$ contains at least $N/4$ zeroes. And for any subinterval of the form $(t_2,u)$, containing $n_u$ points from $\sigma(\Theta)$, the box $(t_2,u)\times (0, \sqrt{ND})$ contains at least $n_u/2$ zeros. We know that $t_n \leq t_2 + (n-1)D$ for $n\geq2$. Thus, enumerating the zeros inside $\tilde S$, we have $x_n \leq t_2 + 2(n-1)D$.

\noindent Thus,
\begin{eqnarray*}
       \pi>\displaystyle\int_{\tilde c}^{t_2} \phi'(t)dt \geq \displaystyle\int_{\tilde c}^{t_2} \displaystyle\sum_{\te} \frac{y_n }{(t-x_n)^2 + y_n^2}dt  
        \geq  \displaystyle\int_{\tilde c}^{t_2} \displaystyle\sum_{n=1}^{N/4} \frac{y_n}{(t-(t_2+nD) )^2+ y_n^2} dt
        &\gtrsim& \sum_{n=1}^{N/4} \displaystyle\int_{\tilde c}^{t_2} \displaystyle \frac{1}{(t-(t_2+nD))^2+ 1}dt \\
        &=& \sum_{n=1}^{N/4}\arctan(Dn) - \arctan(Dn+t_2-\tilde{c}) \\
        &=& \sum_{n=1}^{N/4}  \arctan\bigg(\frac{t_2 - \tilde{c}}{(Dn+t_2 -\tilde{c})Dn}\bigg) \\
        &\geq& \sum_{n=1}^{N/4} \arctan\bigg(\frac{1}{(2Dn/N^2+1)Dn}\bigg)\\
        &\gtrsim& \frac{1}{D}\sum_{n=1}^{N/4} \frac{1}{n},
        \end{eqnarray*}
which diverges as $N\rightarrow\infty$, which is a contradiction.
\end{proof}

Baranov remarks in \cite{BAR} that since the placement of 'other points' does not affect the calculations above, we can make such clusters and  gaps along a very rare subsequence of $\mathbb N$, without affecting the regularity. For example, let us consider the following sequence $\Lambda= \mathbb N \setminus A$, where $A = \{2^{n_k}+m\}$ for $m=1,2,...,k$, where $n_k$ is a rare subsequence of $\mathbb N$, say the sequence $n_k= 3^k$. Then, we have gaps of length $k$, which is unbounded, followed by clusters with gaps of size $1$, the size of the clusters $\geq 2^{n_k+1}$. This sequence is \textit{a-regular}, where $a=1$. For,
\begin{equation*}
\int_{\mathbb R}\frac{|n_{\Lambda}(x) - x|}{1+x^2}dx \asymp \sum_k \frac{k}{1+(2^{n_k})^2} < \infty.
\end{equation*}
Thus, even for regular sequences, there may not exist any MIF with bounded derivative.

           \end{document}